\documentclass[titlepage,draft,11pt]{article} 
\usepackage[title]{appendix}
\usepackage{amssymb,amsthm} 
\usepackage[a4paper]{geometry}
\usepackage[usenames,dvipsnames]{xcolor}

\date{}

\newtheorem{thm}{Theorem}[section]

\newtheorem{rmk}[thm]{Remark}
\newtheorem{prop}[thm]{Proposition}

\newtheorem{hyp}[thm]{Hypothesis}

\title{{A simple characterization of positivity preserving semi-linear parabolic systems }}

\author{Alain Haraux\vspace{1ex}\\ 
{\normalsize Sorbonne Universit\'es, UPMC Univ Paris 06}\\
{\normalsize CNRS, UMR 7598}\\
{\normalsize Laboratoire Jacques-Louis Lions}\\ 
{\normalsize 4, place Jussieu 75005}\\ 
{\normalsize PARIS (France)}\\  
{\normalsize e-mail: \texttt{haraux@ann.jussieu.fr}}}

\begin{document}
\maketitle
\begin{abstract} We give a simple and direct proof of the characterization of positivity preserving semi-flows for ordinary differential systems. The same method provides an abstract result on a class of evolution systems  containing reaction-diffusion systems in a bounded domain  of $ \mathbb{R}^n$ with either Neumann or Dirichlet homogeneous boundary conditions. The conditions are exactly the same with or without diffusion. A similar approach gives the optimal result for invariant rectangles in the case of Neumann conditions.

\vspace{6ex}

\noindent{\bf Mathematics Subject Classification 2010 (MSC2010):}
34C10, 34C14,  35B51, 35K57, 35K58, 35K90, 35K91, 35Q92.

\vspace{6ex}


\end{abstract}

 
\section{Introduction}In chemistry and biology, we often have to deal with evolution systems whose solutions are essentially positive (component concentration rates, population densities, ...). In such cases a model will be disqualified if it predicts negative values, and it is therefore very important to have a criterion for positivity preservation in the forward direction.  The case of a single equation is simple because semi linear parabolic equations are well known to be order preserving, just like first order scalar equations. Therefore the semi-flow of a parabolic equation in a bounded domain $\Omega$ of the form $$\frac{\partial u}{\partial t} - d\Delta u = f(u) \quad \hbox{in}\quad  \mathbb{R}^+\times \Omega $$ where $d>0$ with either Neumann or Dirichlet boundary conditions will be positively preserving if and only if $f(0)\ge 0$. In contrast, even the simple ODE system $ (u', v') = (v, -u) = F(u, v) $ is not positivity preserving although $F(0, 0) = (0, 0).$ \\

In 1970 a simple characterization of closed invariant sets under the semi-flow generated by any locally Lipschitz vector field in $ \mathbb{R}^n$ was given by H.Brezis \cite{Brezis}. This characterization, applicable without any restriction on the regularity of the invariant set $S$, says essentially that at any point of $\partial S$, the vector field is not strictly outgoing. For instance if $S$ is convex, the condition is equivalent to $$\forall u \in \partial S,  F(u) \in \overline { \bigcup _{\lambda>0} \lambda (S-u)}$$ Moreover, if $S$ is a convex polyhedron, it is in fact sufficient to fulfill the condition at those points $u$ for which the tangent cone $\overline { \bigcup _{\lambda>0} \lambda (S-u)}$ is a half-space, since these points are dense in the boundary and the tangent cone of an extremal point is the intersection of the tangent half-spaces of neighboring non-extremal points. These considerations are applicable to the positive cone of $ \mathbb{R}^n$  and would give the correct necessary and sufficient conditions. \\

In this paper, we give a simple and direct proof of the characterization of positivity preserving semi-flows for ordinary differential systems. Then we extend our method to an abstract result on a class of semi linear evolution problems containing the case of reaction-diffusion systems in a bounded domain  of $ \mathbb{R}^n$ with either Neumann or Dirichlet boundary conditions. The conditions will turn out in both cases to be the same with or without diffusion.\\

Apart from the relevance of the models in chemistry and biology, the positivity preserving property is also of major importance in the mathematical investigation of global existence of solutions, since the systems met in biology or chemistry have often the property of conservation of total mass. This led to many research papers, such as \cite{Al, Masuda, Bara, HK, HMP, HY}, and the subject is still active today.\\

The plan of the paper is as follows: Section 2 is devoted to a simple direct proof in the case of ODE systems, in Section 3 a general class of semi linear evolution systems with diagonal linear part is considered, with application to parabolic equations with Neumann or Dirichlet homogeneous boundary conditions, respectively in Sections 4 and 5. The final section 6 is devoted to a few remarks, including a characterization of (positively) invariant rectangles for semi linear parabolic systems with Neumann   homogeneous boundary conditions.

\section{A simple characterization of positivity preserving semi-flows for ordinary differential systems}
\label{Pos.pres.}

 Let $F = (f_1...f_n)$ be a locally Lipschitz continuous vector field on $ \mathbb{R}^n$. We look for a condition to insure positivity preserving of the (local) positive semi-flow generated by the system 
 
 $$ \dot{U} = F(U) $$
 
 \noindent  A first remark is that if  $ f_i (u_1,...,u_{i-1}, 0,u_{i+1},...,u_n) < 0 $ for a certain index  $ i $ and some $(n-1)$-vector $(u_1,... ,u_{i-1},u_{i+1},... ,u_n)$  with all its components  $\ge 0$, the local solution
with initial data  $ (u_1,... ,u_{i-1}, 0, u_{i+1},... ,u_n) $  will have a negative $ i-th$ component for t small. By continuity, the same property will be true for initial data of the form $ (a, a...a)+ (u_1,... ,u_{i-1}, 0, u_{i+1},... ,u_n)$ with  $a>0$ small, so that the flow will not preserve positivity. A necessary condition for positivity preserving is therefore that for all indices $i\in \{1,2...n\}$ we have 
\begin{equation}  \label{pos-cond} \forall (u_j)_{j\not=i} \in( \mathbb{R^+})^{n-1}, \quad f_i (u_1,...,u_{i-1}, 0,u_{i+1},...,u_n) \ge 0  \end{equation}
Surprisingly enough it turns out that this condition is in fact sufficient. We have

\begin{prop}\label{pos.pre}  Assuming condition (\ref{pos-cond}), let $U = (u_1,...u_n)\in C^1([0, Tmax),\mathbb{R}^n)$ be any maximal positive trajectory of the system $\dot{U} = F(U)$. Then if  the initial vector $U(0)$ has all its components positive, the same property is true for the solution vector $U(t)$ for all $ t\in [0, T_{\max}).$ \end{prop}

\begin{proof} Assuming the contrary, let $i$ be one of the indices (there may be several) for which  $u_i(t)$ vanishes for the first time at $ T<T_{\max}$ while the other components have remained $ \ge 0$ until time $T$ . For all  $t \in  (0, T)$:

$$ \dot{u_i}(t) = f_i (u_1,..., u_i,...u_n)=  $$  
$$f_i (u_1,..., u_i,...u_n)- f_i (u_1,...,0,...u_n)+ f_i (u_1,..., 0,...,u_n) \ge -Mu_i ,$$
where  M is a  Lipschitz constant of $ F $ on a product of bounded intervals containing  $U([0, T])$ as well as its projections on the hyperplane  ${ x_i = 0}$.
Then  $ u_i(t) \exp(Mt) $ is nondecreasing on $(0, T)$, in particular by  continuity at $T$ we obtain 
$$u_i(T)\ge \exp(-MT)u_i(0)>0, $$ a contradiction.\end{proof}

\begin{rmk} \begin{em} If $n = 1$ ,  condition (\ref{pos-cond}) reduces to the obvious inequality  $f(0)\ge 0$. \end{em}\end{rmk}
\begin{rmk} \begin{em} Proposition \ref{pos.pre} implies the following more general statement: \end{em}
\begin{thm}\label{pos.pre-we}  Assuming condition (\ref{pos-cond}), let $U = (u_1,...u_n)\in C^1([0, Tmax),\mathbb{R}^n )$ be any maximal positive trajectory of the system $\dot{U} = F(U)$. Then if  the initial vector $U(0)$ has all its components non-negative, the same property is true for the solution vector $U(t)$ for all $ t\in [0, Tmax).$ Conversely if this property holds true for all solutions, then condition (\ref{pos-cond}) is satisfied. Finally, for any $j\in \{1,...n\}$ such that $u_j(0)>0$, we have $$ \forall t\in[0, T_{\max}),\quad u_j(t)>0. $$\end{thm}\begin{proof} Given any $T<T_{\max}$, it is classical that the solution $U_\varepsilon$  with initial condition $U_\varepsilon(0): = U(0)+ \varepsilon(1,...1) $ exists on $[0, T]$ for  $\varepsilon>0$ small enough and $U_\varepsilon$ converges to $U$ in $C^1([0, T])$ as $\varepsilon$ tends to $0$ . The non-negative character of all components follows immediately from \ref{pos-cond}. The converse follows from the preliminary remark. The last property follows from the inequality $$u_j(t)\ge \exp(-Mt)u_j(0), $$ where  M is a  Lipschitz constant of $ F $ on a product of bounded intervals containing  $U([0, T])$ as well as its projections on the hyperplane  $\{ x_j = 0\}$.\end{proof}\end{rmk}

\begin{rmk}\begin{em}\label{nonaut}  Let $F(t, U) = (f_1(t, u_1)...f_n(t,.u_n ))$ be a time dependent vector field on $ \mathbb{R}^n$, Lipschitz continuous on bounded sets in $U$ uniformly for $t\ge0$. Assuming the condition \begin{equation}  \label{nonautpos-cond}\forall t\ge0, \quad   \forall (u_j)_{j\not=i} \in( \mathbb{R^+})^{n-1}, \quad f_i (t, u_1,...,u_{i-1}, 0,u_{i+1},...,u_n) \ge 0  \end{equation}
 let $U = (u_1,...u_n)$ be any maximal positive trajectory of the equation  $\dot{U} = F(t, U(t))$. Then if  the initial vector $U(0)$ has all its components positive, the same property is true for the solution vector $U(t)$ for all $ t\in [0, Tmax).$ However condition  (\ref {nonautpos-cond} ) is not quite necessary in general for positivity preserving of the local semi-flow starting at $0$ . For instance let $g(t)\in C^1(\mathbb{R}^+)$ be any function such that $g (t)>0$ for all $t> 0$  and $g'(t_0)<0$ for some $t_0>0$. Then the scalar equation 
 $$ u' = g'(t) $$gives an example of a system for which $f(t, u) = g'(t)$ does not fulfill condition (\ref{nonautpos-cond}) and, however, the condition $u(0)\ge 0$ implies $u(t) = u(0) + g(t)>0$ for all $t> 0$. Here all solutions are global and preservation of positivity is no longer  true if we start from $t_0$ instead of $0$. \end{em}\end{rmk}
 
 \section{Positivity preserving semi-flows for some abstract semilinear systems}
\label{Parab.}

 Let $d$ any positive integer, let $\Omega$ be a bounded domain of  $ \mathbb{R}^d$ with $C^1$ boundary and let $X$ be a closed subspace of $ C(\overline{\Omega})$. Let $F = (f_1...f_n)$ be a locally Lipschitz continuous vector field on $ \mathbb{R}^n$ and let us consider the diagonal operator $L U= (L_1u_1,...L_n u_n) $ associated to a n-vector  $(L_1, ...L_n)$ of (possibly unbounded) linear operators on $X$.  We assume that each operator $L_i$ is m-dissipative on $X$. We look for a condition to insure positivity preserving of the (local) positive semi-flow generated by the system 
 
 $$ \dot{U} = LU+ F(U) $$ 
 The positive trajectories of the system are by definition the maximal mild solutions of the integral equation 
 
  $$ U(t) = S(t) U(0)+ \int_0^t S(t-s) F(U(s)) ds $$ or, in a more analytic component-wise form 
  $$ u_i(t, x) = [S_i(t) u_i(0)](x) + \int_0^t S_i(t-s) f_i(u_1(s, x),...u_n(s, x)) ds, $$
  where for each $i\in {1,...n}$, we denoted by $S_i(t) = \exp(tL_i)$ the semi-group generated on $X$ by $L_i$
  and $S(t) V: = (S_1(t)v_1,...S_n(t)v_n)$ for each $V= (v_1,...v_n)\in X^n$. \\

 In order to guarantee that the property holds true when $F = 0$, we assume that for each $i\in {1,...n}$, the semi-group $S_i(t) = \exp(tL_i)$ is strongly positive in the following sense: 
 \begin{hyp} \label{strongpos} Whenever $\phi \in X$ satisfies both properties 
$\phi\ge 0$ and $\phi\not\equiv 0$ , we have $$ \forall t>0, \quad \forall x \in \Omega, \quad [S_i(t)\phi ](x)>0 $$  \end{hyp}

We now state our main result

 \begin{thm}\label{strongpos.pre}  Assuming condition (\ref{pos-cond}), let $U= (u_1(t, x),...u_n(t, x) )\in C([0, T_{max}, X^n) $ be any (maximal) positive trajectory of the system  $$ \dot{U} = LU+ F(U) $$ Then if  the initial vector $U(0)= U_0= (u_{0,1}( x),...u_{0,n}( x) )$ is such that for each $i\in {1,...n}$, we have $\forall x \in \Omega, u_{0,i}(x)\ge 0 $ and $u_{0,i}\not\equiv 0 $, then all components of  the solution vector $U(t, x)$ are strictly positive everywhere in $\Omega$ for all $ t\in [0, T_{max})$, assuming that this property is true for $t>0$ small enough.\end{thm}

\begin{proof} Assuming the contrary, let $i$ be one of the indices (there may be several) for which  $u_i(t)$ vanishes for the first time at $ T<T_{max} $ and some point $x\in \Omega$ while the other components have remained positive (at least nonnegative) until time $T$. Then by the same argument as in the ODE case we obtain, by boundedness of the solution on $[0,T]\times \overline{\Omega}$, the existence of a finite positive constant $M$ such that for all  $t \in  (0, T)$:
$$u_i(T, x)\ge[ \exp(T(L_i -M))u_{0, i}](x)>0, $$ a contradiction.\end{proof}
\begin{rmk} \begin{em}Theorem \ref{strongpos.pre} is very general, in particular the operators $L_i$ are not supposed to be symmetric and can be quite independent from each other. In the two next sections we apply this result to parabolic systems of much more restricted type.\end{em}\end{rmk}

 \section{The Neumann case}
\label{Neu.} In this section we investigate the positivity preserving property for a reaction-diffusion system with diagonal diffusion part of the form 
\begin{equation}\label {parab.}  \frac{\partial u_i}{\partial t} - d_i \Delta u_i = f_i(u_1,...u_n) \quad \hbox{in}\quad  \mathbb{R}^+\times \Omega \end{equation}
with all coefficients $d_i>0$ under the homogeneous Neumann boundary conditions \begin{equation} \label {Neu} \frac{\partial u_i}{\partial \nu} = 0\quad \hbox {on}\quad  \mathbb{R}^+\times \partial \Omega.  \end{equation} As a first observation, we remark that for each $i$, the operator $d_i\Delta $ with  homogeneous Neumann boundary conditions and the standard associated domain generates on $ C(\overline{\Omega})$ a contraction semi-group  $ S_i(t) $  which satisfies a strengthened variant of condition \ref{strongpos}: Whenever $\phi \in X$ satisfies both properties 
$\phi\ge 0$ and $\phi\not\equiv 0$ , we have $$ \forall t>0, \quad \inf_{x\in\overline{\Omega} } [S_i(t)\phi ](x)>0 $$  The second observation is that the spatially homogeneous solutions $u_i(t, x) = u_i(t)$ of the ODE \begin{equation}  \dot {u_i}= f_i(u_1,...u_n) \quad \hbox{in}\quad  \mathbb{R}^+ \end{equation} are particular solutions of the problem since the boundary conditions are satisfied. From these two remarks and Theorem \ref{strongpos.pre}, it is easy to deduce the following result

 \begin{thm}\label{Neupos.pre}  Assuming condition (\ref{pos-cond}), let $U(t, x) = (u_1(t, x),...u_n(t, x) )$ be any maximal positive trajectory of the system (\ref{parab.}) with boundary conditions \ref{Neu}. Then if  the initial vector $U(0)= U_0= (u_{0,1}( x),...u_{0,n}( x) )$ is such that for each $i\in {1,...n}$, we have $\forall x \in \Omega, u_{0,i}(x)\ge 0 $, then all components of  the solution vector $U(t, x)$ are non-negative everywhere in $\Omega$ for all $ t\in [0, T_{max})$. Conversely is this property is satisfied for all positive initial vectors $U_0= (u_{0,1}( x),...u_{0,n}( x) )$, then F has to satisfy condition (\ref{pos-cond}).  Finally whenever $j\in \{1,...n\}$ is such that $u_j(0,.)\not=0$, we have $$ \forall t\in(0, T_{\max}),\quad \inf_{x\in\overline{\Omega} }u_j(t,x)>0. $$\end{thm}

\begin{proof}  The second part is obvious since we can apply the first remark on ODE systems to spatially homogeneous solutions. For the first part, we start by replacing 
$U_0= (u_{0,1}( x),...u_{0,n}( x) )$ by $U^\varepsilon_0= (u_{0,1}( x)+ \varepsilon,...u_{0,n}( x) +\varepsilon)$ and consider the corresponding maximal solution $U^\varepsilon(t, x) = (u^\varepsilon_1(t, x),...u^\varepsilon_n(t, x) )$ of the system (\ref{parab.}) with boundary conditions \ref{Neu}. It is classical (cf e.g. \cite{CH}, Proposition 4.3.7. p.59) that the life time $T^\varepsilon_{max}$of $U^\varepsilon$ overpasses any number $ T<T_{max} $ for $\varepsilon$ small enough and  the sequence of vector functions $U^\varepsilon(t, x) = (u^\varepsilon_1(t, x),...u^\varepsilon_n(t, x) )$ converges uniformly to $U(t, x)$  in $C([0, T] \times \overline{\Omega})$ as $\varepsilon$ tends to $0$.  Now clearly, for any $\varepsilon>0$, the solution  $U^\varepsilon(t, x)$ has, by continuity, all its components positive in $\overline{\Omega}$ for $t$ small. Theorem \ref{strongpos.pre} applied to $U^\varepsilon$ now shows that the solution  $U^\varepsilon(t, x)$ has all its components positive in $\Omega$ for all $t\in [0, T]$. We conclude by letting $\varepsilon$ tend to $0$. The last statement follows from the inequality $$\frac{\partial u_j}{\partial t} - d_j \Delta u_i = f_i(u_1,...u_n) \ge -M u_j\quad \hbox{in}\quad  (0, T]\times \Omega  $$ which implies $u_j(t, x) \ge \exp(-Mt) S_j(t) \psi (x)$ where $\psi(x): = u_j(0,x)$
\end{proof}

 \section{The Dirichlet case}\label{Dir.} In this section we assume that $\Omega$ is connected with a $C^2$ boundary and we investigate the positivity preserving property for a reaction-diffusion system with diagonal diffusion part of the form 
\begin{equation}  \frac{\partial u_i}{\partial t} - d_i \Delta u_i = f_i(u_1,...u_n) \quad \hbox{in}\quad  \mathbb{R}^+\times \Omega \end{equation}
with homogeneous Dirichlet boundary conditions \begin{equation} \label{Dir.} u_i= 0\quad \hbox {on}\quad  \mathbb{R}^+\times \partial \Omega.  \end{equation}
 As a first observation, we remark that for each $i$ , the contraction semi-group $ S_i(t) = \exp(t d_i\Delta)$ on $X = C_0(\Omega)= \{u\in C(\overline{\Omega}) / u=0  \,\,\hbox {throughout}\,\, \partial\Omega $ associated with the homogeneous Dirichlet boundary conditions satisfies  (\ref{strongpos}). From this remark and Theorem \ref{strongpos.pre}, we deduce
 \begin{thm}\label{Dirpos.pre}  Assuming condition (\ref{pos-cond}), let $U(t, x) = (u_1(t, x),...u_n(t, x) )$ be any maximal positive trajectory of the system (\ref{parab.}) with boundary conditions \ref{Dir.}. Then if  the initial vector $U(0)= U_0= (u_{0,1}( x),...u_{0,n}( x) )$ is such that for each $i\in {1,...n}$, we have $\forall x \in \Omega, u_{0,i}(x)\ge 0 $, then all components of  the solution vector $U(t, x)$ are non-negative everywhere in $\Omega$ for all $ t\in [0, T_{max})$.  Moreover whenever $j\in \{1,...n\}$ is such that $u_j(0,.)\not=0$, we have $$ \forall t\in(0, T_{\max}),\,\,\forall x\in\Omega, \quad u_j(t,x)>0. $$ \end{thm}

\begin{proof}  We follow the scheme of proof of the previous result  by replacing $U_0$ by $U^\varepsilon_0= (u_{0,1}( x)+ \varepsilon\varphi_1(x),...u_{0,n}( x) +\varepsilon\varphi_1(x))$ where $\varphi_1$ is the normalized positive eigenfunction associated to the first eigenvalue $ -\Delta$ in $H^1_{0}(\Omega)$  and consider the corresponding maximal solution $U^\varepsilon(t, x) = (u^\varepsilon_1(t, x),...u^\varepsilon_n(t, x) )$ of the system (\ref{parab.}) with boundary conditions \ref{Dir.}. Here also, the life time $T^\varepsilon_{max}$of $U^\varepsilon$ overpasses any number $ T<T_{max} $ for $\varepsilon$ small enough and  the sequence of vector functions $U^\varepsilon(t, x) = (u^\varepsilon_1(t, x),...u^\varepsilon_n(t, x) )$ converge uniformly to $U(t, x)$  in $C([0, T] \times \overline{\Omega}$ as $\varepsilon$ tends to $0$.  \\

At this point we need to check that for any $\varepsilon>0$, the solution  $U^\varepsilon(t, x)$ has all its components positive in $\Omega$  for $t$ small. Actually we have for all $j$ the formula $$ u^\varepsilon_j(t, x) = S_j(t)(u_{0,j}( x)+ \varepsilon\varphi_1(x)) + \int _0^t S_j(t-s)f_j(u^\varepsilon_1(s, x),...u^\varepsilon_n(s, x) ) ds $$ Using boundedness of the components and the smoothing effect of the heat semigroup from $X = C_0(\Omega)$ to $ C^1(\overline{\Omega})$ such as described in \cite{HK}, Theorem 1.1, we easily find for some constant $C$ depending on the solution the estimate valid for $t$ small:
$$ u^\varepsilon_j(t, x) \ge \varepsilon \exp( -d_j \lambda_1 t)\varphi_1(x) - Ct^{1/3}\varphi_1(x) .$$ From this formula it is clear that for $t$ small enough
$$ u_j(t, x) \ge \frac{\varepsilon}{2}\varphi_1(x) $$ The regularity of $\partial \Omega$ is needed here since we use the estimate $$|w(x)| \le C_1||\nabla w||_{\infty} dist(x, \partial\Omega) \le C_2 ||w||_{C^1(\overline{\Omega})}\varphi_1(x) $$ to derive the bound $$| \int _0^t S_j(t-s)f_j(u^\varepsilon_1(s, x),...u^\varepsilon_n(s, x) ) ds | \le K( U_\varepsilon) \int _0^t (t-s)^{-2/3}\varphi_1(x) ds$$ for $t$ small enough depending on $\varepsilon$.
Theorem \ref{strongpos.pre} applied to $U^\varepsilon$ now shows that the solution  $U^\varepsilon(t, x)$ has all its components positive in $\Omega$ for all $t\in [0, T]$ . We conclude by letting $\varepsilon$ tend to $0$. The proof of the last property is the same as in the Neumann case. \end{proof}
The converse of Theorem \ref{Dirpos.pre} is also true, although more delicate to prove. More precisely we have

 \begin{thm}\label{Dir-Conv} Assume conversely that whenever the initial vector
 $U(0)= U_0\in [C_0(\Omega)]^n$ is such that for each $i\in {1,...n}$, we have $\forall x \in \Omega, u_{0,i}(x)\ge 0 $, then all components of  the solution vector $U(t, x)$ are non-negative everywhere in $\Omega$ for all $ t\in [0, T_{max})$.Then condition (\ref{pos-cond}) is satisfied. \end{thm}

\begin{proof}  We reason by contradiction. Let us suppose that  $$ f_i (u_1,...,u_{i-1}, 0,u_{i+1},...,u_n) = -\beta< 0 $$ for a certain index  $ i $ and some $(n-1)$-vector $(u_1,... ,u_{i-1},u_{i+1},... ,u_n)$  with all its components  $\ge 0$. First of all we observe that for any $\varepsilon>0$ the system of equations  $$\frac{\partial u_j}{\partial t} - d_j\Delta u_j = f_j(u_1,...u_n) \quad \hbox{in}\quad  \mathbb{R}^+\times \Omega $$ with homogeneous Dirichlet boundary conditions has a mild local solution $U$ with initial value $(u_1,... ,u_{i-1},0, u_{i+1},... ,u_n)  $ which is defined on a small time interval $ [0, \delta]$, uniformly bounded on $\Omega$  and  continuous with values in $ L^2 (\Omega)$ .   Moreover this property is in fact true for any initial data in $ [L^\infty (\Omega]^n)$ with a fixed existence time $\delta$ depending only on the initial norm in  $ [L^\infty (\Omega]^n)$. More precisely, if $M>0$ is a bound of the given initial norm and $L = L(R)$ is a Lipschitz norm of $F$ in the n-product of balls of centre $0$ and radius $R>M$, the number $\delta$ just needs to satisfy the condition  $ M + \delta ( ||F(0)||+LR) < R$  for some $R>M$, i.e. $ \delta < \frac{R-M}{||F(0)||+LR} $ . For $R= 2M$ this amounts to $ \delta < \frac{M}{||F(0)||+2LM} $. We skip the details but now the important thing is that we can approach  the initial conditions $(u_1,... ,u_{i-1},0, u_{i+1},... ,u_n)  $ from below by a sequence of vectors $U^k \in X^n = [C_0(\Omega)]^n$ with all components non-negative, keeping the component of index $i$ equal to $0$. The component of index $i$ satisfies the integral equation 
$$ u^k_i(t, x) =  \int _0^t S_i(t-s)f_i(u^k_1(s, x),0, ...u^k_n(s, x) ) ds $$ On the small interval $ [0, \delta]$, the vector $U^k(t, x)$ tends to $U(t, x)$ in $C([0, \delta], L^2)$. Therefore  $f_j(u^k_1(s, x),0, ...u^k_n(s, x)$ also tends to $f_i(u_1(s, x),0, ...u_n(s, x)$ in $C([0, \delta], L^2)$. Hence, taking the inner product in $ L^2 (\Omega)$ by the positive normalized first  eigenfunction $ \varphi_1 $ of the Dirichlet- Laplacian and using the self-ajoint character of the linear semigroup $S_i(t) $ we obtain$$ \int_{\Omega} u^k_i(t, x) \varphi_1 (x) dx  =  \int _0^t \int_{\Omega}f_i(u^k_1(s, x),0, ...u^k_n(s, x) ) S_i(t-s)\varphi_1 (x)dx ds$$ whence   $$ \lim_{k\rightarrow \infty}\int_{\Omega} u^k_i(t, x) \varphi_1 (x) dx =\int _0^t \int_{\Omega}f_i(u_1(s, x),0, ...u_n(s, x) ) S_i(t-s)\varphi_1 (x)dxds $$ Finally the last term is equal to $$ \int _0^t \int_{\Omega}f_i(u_1(s, x),0, ...u_n(s, x) ) \exp(-d_i\lambda_1(t-s))\varphi_1 (x)dx ds $$ or 
$$ \exp(-t d_i\lambda_1)\int _0^t \int_{\Omega}f_i(u_1(s, x),0, ...u_n(s, x) ) \exp(-d_i\lambda_1s)\varphi_1 (x)dxds $$ But due to the continuity we now have 
$$ \lim_{t\rightarrow 0}\frac{1}{t}\int _0^t \int_{\Omega}f_i(u_1(s, x),0, ...u_n(s, x) ) \exp(-d_i\lambda_1s)\varphi_1 (x)dxds = -\beta \int_{\Omega}\varphi_1 (x)dx<0 $$
 This contradiction concludes the proof \end{proof}

 \section{Concluding remarks, invariant rectangles} It is clear that in our main results on semilinear equations, the diffusion operators $ L_i = d_i \Delta$ can be replaced by any strongly  elliptic differential operator of order 2 with smooth coefficients in divergence form , and they do not need to be multiples of the same operator. Moreover, some of the results can be generalized to non-autonomous systems of the form $$ \dot{U} = LU+ F(t, U) $$ but, as already mentionned in the case of ODE systems, the converse statements on $F$ will not be true in general. Finally some of the results should be also valid in unbounded domains. We end the paper by a few simple observations on related topics.

 \begin{rmk}\label{Brez2} \begin{em}The application of \cite{Brezis} to positivity preserving of finite dimensional systems is quite immediate since the  positive cone in $ \mathbb{R}^n$ is a particular polyhedron with a very simple boundary. However the direct proof is quite simple and gives a framework easily applicable to parabolic problems. \end{em}\end{rmk}
 
  \begin{rmk}\label{Brez} \begin{em} The characterization by Brezis (cf.\cite{Brezis}) of positive invariance is also valid in general Banach spaces. Analogous formulations have been used for general evolution equations, cf. e.g. \cite{M, Pavel}. \end{em} \end{rmk}
 
 \begin{rmk}\label{Inv-sets} \begin{em} The criterion from \cite{Brezis} has been extended  to find other invariant sets, in particular the obtention of invariant sets of the form $a\le u\le b$ provides immediately global existence results for some initial data. In the case of systems, the method of sub and super solutions is not applicable to the full system and invariant regions become very useful.  Actually in the case of Neumann boundary conditions, general invariant rectangles can be characterized easily as a consequence of the method of proof of Theorem \ref{Neupos.pre}  \end{em} \end{rmk}
 
 Let us consider two sequences of n real numbers $ (\alpha_1, ...\alpha_n)$ and  $(\beta_1, ...\beta_n$ with $\alpha_j<\beta_j$ for all $ j\in \{1,...n\}$. Then we have 
 
 \begin{thm}\label{Invariant rectangles}  The region $ {\displaystyle  R = \prod_{1\le j\le n} [\alpha_j, \beta_j] }$ is invariant under the (local) semi flow generated by the  system 
 $ \dot{U} = F(U) $ if, and only if the two following conditions are fulfilled
  \begin{equation}  \label{Rinf} \forall (u_j)_{j\not=i} \in\prod_{j\not =i} [\alpha_i, \beta_j] , \quad f_i (u_1,...,u_{i-1}, \alpha_i,u_{i+1},...,u_n) \ge 0  \end{equation}\begin{equation} 
  \label{Rsup} \forall (u_j)_{j\not=i} \in\prod_{j\not =i} [\alpha_j, \beta_j] , \quad f_i (u_1,...,u_{i-1}, \beta_i,u_{i+1},...,u_n) \le 0  \end{equation} Moreover the rectangle ${\cal R} = \{U\in C(\overline{\Omega}, {\mathbb R}^n), \quad \forall x \in \overline{\Omega}, \,U(x) \in R \} $ is invariant under the local semi flow generated by the parabolic system (\ref{parab.}) with  homogeneous Neumann boundary conditions (\ref {Neu}) if, and only if the conditions ( \ref{Rinf}) and (\ref{Rsup}) are both satisfied.

\end{thm}

\begin{proof} The necessary condition is obvious. The proof of sufficient condition is an  easy adaptation of the proof of Proposition \ref{pos.pre} and \ref{Neupos.pre} considering the various possible ways for a solution to  escape the rectangle. \end{proof}

\begin{rmk} \begin{em}The strict versions of (\ref{Rinf}) and (\ref{Rsup}) were used in \cite{BM} to build invariant rectangles. \end{em} \end{rmk}

\begin{rmk} \begin{em} Since parabolic equations with Dirichlet boundary conditions are not invariant through translation by constants, we have no analog of 
Theorem \ref{Invariant rectangles} in that case and the search for invariant regions seems to become more complicated. \end{em} \end{rmk}

 {\small

\end{document}